\documentclass[a4paper,11pt]{article}

\usepackage[margin=25mm]{geometry}
\usepackage{amsmath,amssymb,amsthm,mathtools,bm}
\usepackage{booktabs,array}
\usepackage{graphicx}
\usepackage{hyperref}
\usepackage{enumitem}
\allowdisplaybreaks[3]

\hypersetup{
  colorlinks=true,
  linkcolor=blue,
  citecolor=blue,
  urlcolor=blue,
  pdftitle={A Structure-Preserving Stagewise Rescaling Algorithm for a Two-Dimensional Nonlocal MEMS Equation in an Asymptotically Constant-Feedback Regime},
  pdfauthor={Takiko Sasaki and Tetsuji Tokihiro}
}

\theoremstyle{plain}
\newtheorem{theorem}{Theorem}[section]
\newtheorem{proposition}[theorem]{Proposition}
\newtheorem{lemma}[theorem]{Lemma}

\newtheorem{assumption}[theorem]{Assumption}

\theoremstyle{definition}
\newtheorem{definition}[theorem]{Definition}
\newtheorem{remark}[theorem]{Remark}

\numberwithin{equation}{section}

\title{A Structure-Preserving Stagewise Rescaling Algorithm\\
for a Two-Dimensional Nonlocal MEMS Equation\\
in an Asymptotically Constant-Feedback Regime}
\author{%
Takiko Sasaki$^{1}$ and Tetsuji Tokihiro$^{1}$\\[1ex]
\small $^{1}$Department of Mathematical Engineering, Faculty of Engineering, Musashino University\\
\small Corresponding author: Takiko Sasaki, \texttt{t-sasaki@musashino-u.ac.jp}%
}
\date{}

\begin{document}
\maketitle

\begin{abstract}
Nonlocal MEMS equations exhibit finite-time quenching phenomena that pose significant challenges for numerical simulation. In this paper, we study a stagewise rescaling algorithm for a two-dimensional nonlocal MEMS equation in an asymptotically constant-feedback touchdown regime. The original nonlocal equation is not exactly invariant under the $A^{3/2}$--$A^3$ scaling used below; rather, the scaling is appropriate when the reciprocal-integral feedback
\[
K(t)=1+\int_\Omega (1-u(x,t))^{-1}\,dx
\]
remains bounded and converges to a finite positive limit, as in the single-point touchdown profiles constructed by Duong--Zaag. In this regime the leading-order core balance is that of a local MEMS equation with an asymptotically constant coefficient. By applying a fixed-stage scaling to the deficit variable, we transform the equation into a gradient flow for a rescaled energy at a frozen amplitude and obtain an exact energy dissipation identity within each stage. We then introduce a minimizing-movement stage solver and obtain a discrete energy inequality at the fixed-stage level. Because strict energy monotonicity is not expected at stage transitions, we isolate a switch defect and an outer-update defect to prove an exact defect balance. Conditional on a uniform switch-defect estimate, this balance implies quantitative almost monotonicity. We also formulate a defect-aware criterion for the nonexistence of a global admissible continuation. Finally, we reorganize the numerical section around reproducible two-dimensional reference computations: a full-domain stagewise run illustrating trigger detection, fixed-stage energy decay, and geometric accumulation of physical time, and a direct fixed-domain energy check. These computations are not used as a proof of the bounded-window criterion; they include a finite-feedback diagnostic table and identify the additional ideal-transfer switch-energy diagnostics needed for a posteriori verification.
\end{abstract}


\section{Introduction}

Let $\Omega\subset \mathbb{R}^2$ be a bounded domain, and consider the two-dimensional nonlocal MEMS equation
\[
\begin{cases}
u_t-\Delta u=\displaystyle
\frac{\lambda}{(1-u)^2\left(1+\int_\Omega (1-u)^{-1}\,dx\right)^2},
& x\in\Omega,\ t>0,\\[1ex]
u=0,& x\in\partial\Omega,\\
0\le u(x,0)=u_0(x)<1,& x\in\Omega.
\end{cases}
\]
Nonlocal MEMS models with integral feedback of this type arise naturally in electrostatic device control; see, for example, Pelesko--Triolo \cite{PeleskoTriolo}, Guo--Hu--Wang \cite{GuoHuWang}, Guo--Kavallaris \cite{GuoKavallaris}, and the monograph of Kavallaris--Suzuki \cite{KavallarisSuzuki}. For the related local electrostatic MEMS evolution problem and its touchdown dynamics, see Ghoussoub--Guo \cite{GhoussoubGuoDynamic}. The approach of the solution toward the singular level $u=1$ is called \emph{quenching} or \emph{touchdown}; for the nonlocal problem above, global-vs-quenching behavior and finite-time quenching mechanisms have been studied in particular in \cite{GuoHuWang,GuoKavallaris,KavallarisSuzuki}.

Near the quenching core the natural variable is not $u$ but the deficit
\[
v:=1-u>0.
\]
In terms of $v$, the equation becomes
\[
v_t=\Delta v-\frac{\lambda}{v^2K(v)^2},
\qquad
K(v):=1+\int_\Omega v^{-1}\,dx,
\]
with boundary condition $v=1$ on $\partial\Omega$. The associated energy is
\[
E[v]=\frac12\int_\Omega |\nabla v|^2\,dx
+\frac{\lambda}{1+\int_\Omega v^{-1}\,dx},
\]
and any sufficiently smooth solution satisfies
\[
\frac{d}{dt}E[v(t)]=-\int_\Omega v_t(x,t)^2\,dx\le 0.
\]
This dissipation law is one of the main structural ingredients in the analytical study of the nonlocal MEMS equation; see \cite{GuoKavallaris,KavallarisSuzuki}.

Stagewise rescaling algorithms for numerical blow-up computations go back to Berger and Kohn \cite{BergerKohn}; a recent convergence analysis of that algorithm was given by Cho and Sun \cite{ChoSun}. On the structure-preserving discretization side, Matsuya and Tokihiro \cite{MatsuyaTokihiro} showed how a discrete scheme can be designed so as to retain blow-up/global-existence information for a semilinear heat equation. The present paper is motivated by bringing these two directions together for a nonlocal MEMS problem whose energy contains the reciprocal integral $\int_\Omega v^{-1}\,dx$.

A point that is essential for the present algorithm is that the original deficit equation is \emph{not} scale invariant. The transformation
\[
x-x_*=A^{3/2}\xi,\qquad t-t_*=A^3s,\qquad v=AW
\]
is instead tied to the asymptotically constant-feedback regime
\[
K(t):=1+\int_\Omega v(x,t)^{-1}\,dx\longrightarrow K_T\in(0,\infty)
\qquad \text{as } t\uparrow T.
\]
In that case the nonlocal coefficient $\lambda K(t)^{-2}$ converges to a positive constant, and the leading-order balance near an isolated touchdown core is the same as for the local equation $v_t=\Delta v-\alpha_Tv^{-2}$. This regime is realized by the single-point interior touchdown solutions constructed by Duong--Zaag \cite{DuongZaag} for the corresponding nonlocal MEMS model with a feedback parameter. Their final profile has the form
\[
1-u^*(x)\sim C\left(\frac{|x-a|^2}{|\log |x-a||}\right)^{1/3}
\qquad (x\to a),
\]
so, in two dimensions, $(1-u^*)^{-1}$ is locally integrable near the touchdown point. Thus the reciprocal integral can have a finite limiting contribution, and the feedback is asymptotically constant at leading order. The present paper should therefore be read as a structure-preserving numerical framework for this Duong--Zaag type regime, rather than as a scaling theory for every possible nonlocal touchdown scenario.

Both the theoretical development and the reference numerical experiments reported in this manuscript are written for the \emph{two-dimensional} problem. No one-dimensional numerical run is used as evidence for the results below. This separation keeps the reported computations aligned with the two-dimensional nonlocal scaling and with the constant-feedback scope described above.

Our goals are the following.
\begin{enumerate}[label=(\roman*)]
\item Make explicit the asymptotically constant-feedback regime in which the $A^{3/2}$--$A^3$ scaling is appropriate.
\item Preserve the exact fixed-stage energy dissipation under rescaling.
\item Derive a minimizing-movement discretization whose fully discrete fixed-stage energy is nonincreasing.
\item Isolate the loss of stage-to-stage monotonicity into a \emph{switch defect} and, for finite-window implementations, an \emph{outer-update defect}.
\item Formulate a defect-aware criterion that excludes the existence of a global admissible continuation under additional summability and growth assumptions.
\item Align the numerical section with the theory by reporting two-dimensional reference computations, including finite-feedback diagnostics, and by stating which switch-defect diagnostics remain absent.
\end{enumerate}

Compared with the classical Berger--Kohn setting \cite{BergerKohn}, the present equation contains a genuinely nonlocal coefficient, so the stage transition must control not only the local interpolation error but also the change of the reciprocal-integral term. This is why the paper is organized around three distinct levels:
\begin{itemize}
\item the fixed-stage gradient-flow structure;
\item the stage-transition defect balance; and
\item the bounded-window continuation criterion.
\end{itemize}
Within the stated asymptotically constant-feedback scope, the fixed-stage energy identity is unconditional, whereas the quantitative switch-defect estimate remains conditional in the current manuscript.

The main new points of the paper are as follows. First, we separate the scope of the scaling from the structure-preserving discretization: the $A^{3/2}$--$A^3$ scaling is motivated by the Duong--Zaag type constant-feedback core, while the fixed-stage energy identity is proved exactly for each frozen amplitude. Second, for frozen $A$ we write the rescaled equation exactly as a gradient flow for a rescaled energy and retain the corresponding discrete energy inequality under a minimizing-movement scheme. Third, at stage transitions we use a 12-point prolongation that is globally $C^0$ across coarse-cell interfaces and locally compatible with the centered-difference Laplacian away from interface-crossing stencils. Fourth, we separate the stage-jump error into an ideal switch defect and an outer-update defect, which leads to an exact defect balance and, under a conditional estimate of the switch defect, to quantitative almost monotonicity. Finally, in Section~5 we formulate a criterion for the \emph{nonexistence of a global admissible continuation} on uniformly bounded energy boxes. This is deliberately weaker than a finite-time quenching theorem, and we state it in that form to avoid conflating the discrete continuation argument with the continuous quenching theory.

The numerical section should also be read with this distinction in mind. The two-dimensional reference computations are full-domain computations performed on the growing rescaled square, so they illustrate fixed-stage energy decay and geometric time accumulation, but they do \emph{not} verify the bounded-box criterion of Section~5. Section~6 reports the reconstructed feedback diagnostics $K(t)$ and $\lambda K(t)^{-2}$ at the stage endpoints; these data show that the feedback remains finite over the reported stages, but four stages are not enough to prove convergence to an asymptotically constant-feedback regime. Moreover, the reference computations still do not record the ideal-transfer switch-energy diagnostics needed for an a posteriori measurement of the stage-transition defect.

\section{Fixed-Stage Rescaling and Exact Energy Dissipation}

The fixed-stage scaling used in this section should be understood with the scope described in the introduction. The deficit equation
\[
v_t=\Delta v-\frac{\lambda}{v^2K(v)^2},
\qquad
K(v)=1+\int_\Omega v^{-1}\,dx,
\]
is not invariant under the transformation used below: the domain, the boundary data, and the nonlocal factor all change. The scaling is instead based on a dominant-balance argument valid when the feedback is asymptotically constant.

More precisely, suppose that near an isolated touchdown point $x_*$ and time $T$ one has
\[
K(t):=1+\int_\Omega v(x,t)^{-1}\,dx\to K_T\in(0,\infty).
\]
Then
\[
\alpha(t):=\lambda K(t)^{-2}\to \alpha_T:=\lambda K_T^{-2}>0,
\]
and the leading-order core equation is
\[
v_t\simeq \Delta v-\alpha_Tv^{-2}.
\]
If $v\sim A$, the balances
\[
v_t\sim A/\tau,
\qquad
\Delta v\sim A/\ell^2,
\qquad
v^{-2}\sim A^{-2}
\]
give $\ell\sim A^{3/2}$ and $\tau\sim A^3$. This yields
\[
x-x_*=A^{3/2}\xi,\qquad t-t_*=A^3 s,\qquad v=AW.
\]
If instead $K(t)$ had a leading-order dependence on $A$, for instance $K(t)\sim A^{-\beta}$, then the natural time and space scales would generally be modified. The present algorithm therefore targets the asymptotically constant-feedback regime, not an arbitrary nonlocal touchdown scenario.

Throughout this section $A>0$, $x_*$, and $t_*$ are fixed and do not vary continuously within a stage.

\begin{definition}[Full-domain fixed-stage rescaling]
Define
\[
\Omega_A:=\frac{\Omega-x_*}{A^{3/2}},
\]
and
\[
W(\xi,s):=\frac{1}{A}v(x_*+A^{3/2}\xi,t_*+A^3 s),
\qquad (\xi,s)\in \Omega_A\times[0,S].
\]
\end{definition}

Under this change of variables,
\[
v_t=A^{-2}W_s,\qquad \Delta_x v=A^{-2}\Delta_\xi W,\qquad dx=A^3\,d\xi.
\]
Hence the rescaled equation becomes
\[
W_s=\Delta_\xi W-\frac{\lambda}{W^2K_A(W)^2},
\qquad
K_A(W):=1+A^2\int_{\Omega_A}W^{-1}\,d\xi.
\]
The boundary condition $v=1$ becomes
\[
W(\xi,s)=A^{-1},\qquad \xi\in \partial\Omega_A.
\]

We define the rescaled energy by
\[
E_A[W]:=\frac{A^2}{2}\int_{\Omega_A} |\nabla_\xi W|^2\,d\xi
+\frac{\lambda}{1+A^2\int_{\Omega_A}W^{-1}\,d\xi}.
\]

\begin{proposition}[Exact dissipation at a fixed stage]
If $W$ is a smooth solution of the rescaled equation, then
\[
\frac{d}{ds}E_A[W(s)]
=
-A^2\int_{\Omega_A}W_s(\xi,s)^2\,d\xi
\le 0.
\]
\end{proposition}

\begin{proof}
The physical and rescaled energies satisfy
\[
E[v(t_*+A^3 s)]=E_A[W(s)].
\]
Using the energy dissipation law for the physical equation,
\[
\frac{d}{dt}E[v(t)]=-\int_\Omega v_t(x,t)^2\,dx,
\]
we obtain
\[
\frac{d}{ds}E_A[W(s)]
=
A^3\frac{d}{dt}E[v(t)]
=
-A^3\int_\Omega v_t(x,t)^2\,dx.
\]
Since $v_t=A^{-2}W_s$ and $dx=A^3\,d\xi$,
\[
A^3\int_\Omega v_t^2\,dx
=
A^3\int_{\Omega_A}A^{-4}W_s^2A^3\,d\xi
=
A^2\int_{\Omega_A}W_s^2\,d\xi.
\]
Substituting this into the previous identity proves the claim.
\end{proof}

\begin{remark}[No exact scale invariance]
The transformed coefficient
\[
K_A(W)=1+A^2\int_{\Omega_A}W^{-1}\,d\xi
\]
still contains the frozen amplitude $A$, and the boundary condition becomes $W=A^{-1}$ on $\partial\Omega_A$. Thus the transformation is not a symmetry of the original nonlocal problem. Its role is to magnify the touchdown core under the assumption that the nonlocal coefficient is asymptotically constant at leading order. If $A$ were allowed to vary continuously within a stage, transport and dilation terms generated by $\dot A$ would destroy the simple gradient-flow structure. Following the Berger--Kohn type stagewise strategy \cite{BergerKohn,ChoSun}, we therefore freeze $A$ inside each stage and rescale only at stage transitions.
\end{remark}

\section{A Fully Discrete Fixed-Stage Scheme with Energy Dissipation}

For the stagewise discrete scheme we use stage-dependent axis-aligned rectangular boxes
\[
Q_{\rm en}^{(m)}\Subset Q_{\rm int}^{(m)}\subset \mathbb{R}^2.
\]
In a bounded-window implementation these boxes may be independent of $m$. In the full-domain computations of Section~6 we take
\[
Q_{\rm en}^{(m)}=Q_{\rm int}^{(m)}=\Omega_{A_m},
\]
which is again a square for the unit-square example considered there.

The mesh convention is the following. Since
\[
A_{m+1}=k^{-2/3}A_m
\quad\Longrightarrow\quad
L_{m+1}=kL_m
\]
for the side length of the rescaled square, the full-domain stagewise implementation used later enlarges the number of grid intervals by the same factor and keeps the mesh width fixed:
\[
N_{m+1}=kN_m,\qquad h_{m+1}=h_m\equiv h.
\]
The alternative rule $h_{m+1}=h_m/k$ belongs to a different stage-transition discretization and is \emph{not} used in the present manuscript. To keep the notation compatible with the stage index, we nevertheless continue to write $h_m$, with the understanding that in the full-domain formulation $h_m=h$ for every $m$.

For stage $m$ we set
\[
Q_{{\rm en},h_m}:=Q_{\rm en}^{(m)}\cap (h_m\mathbb Z)^2,\qquad
Q_{{\rm int},h_m}:=Q_{\rm int}^{(m)}\cap (h_m\mathbb Z)^2,
\]
and denote by $Q_{{\rm en},h_m}^\circ$ and $\partial Q_{{\rm en},h_m}$ the interior and boundary grid points, respectively. We use the discrete inner product
\[
(Y,Z)_{2,h_m}:=h_m^2\sum_{Q_{{\rm en},h_m}^\circ}Y_{ij}Z_{ij},
\qquad
\|Y\|_{2,h_m}^2:=(Y,Y)_{2,h_m}.
\]

Let $g_m$ denote the prescribed stage-$m$ boundary values on $\partial Q_{{\rm en},h_m}$. In the full-domain fixed-stage setting one has $g_m\equiv A_m^{-1}$. Given an interior grid function $Y$, we write $Y^\flat$ for the extension to $Q_{{\rm en},h_m}$ obtained by setting $Y^\flat=Y$ on $Q_{{\rm en},h_m}^\circ$ and $Y^\flat=g_m$ on $\partial Q_{{\rm en},h_m}$.

We define the forward differences by
\[
D^{+x}_{h_m}Y_{ij}:=\frac{Y^\flat_{i+1,j}-Y^\flat_{ij}}{h_m},
\qquad
D^{+y}_{h_m}Y_{ij}:=\frac{Y^\flat_{i,j+1}-Y^\flat_{ij}}{h_m},
\]
on the index sets
\[
Q_{{\rm en},h_m}^x:=\{(i,j)\in Q_{{\rm en},h_m}\,;\, (i+1,j)\in Q_{{\rm en},h_m}\},
\]
\[
Q_{{\rm en},h_m}^y:=\{(i,j)\in Q_{{\rm en},h_m}\,;\, (i,j+1)\in Q_{{\rm en},h_m}\},
\]
and the discrete gradient norm by
\[
\|\nabla_{h_m}Y\|_{2,h_m}^2
:=
h_m^2\sum_{Q_{{\rm en},h_m}^x}|D^{+x}_{h_m}Y_{ij}|^2
+h_m^2\sum_{Q_{{\rm en},h_m}^y}|D^{+y}_{h_m}Y_{ij}|^2.
\]
The discrete Laplacian on $Q_{{\rm en},h_m}^\circ$ is the standard five-point stencil
\[
(\Delta_{h_m}Y)_{ij}
:=
\frac{
Y^\flat_{i+1,j}+Y^\flat_{i-1,j}+Y^\flat_{i,j+1}+Y^\flat_{i,j-1}-4Y^\flat_{ij}
}{h_m^2}.
\]
For test functions $\Phi$ vanishing on $\partial Q_{{\rm en},h_m}$, the discrete Green identity
\[
(-\Delta_{h_m}Y,\Phi)_{2,h_m}
=
(\nabla_{h_m}Y,\nabla_{h_m}\Phi)_{2,h_m}
\]
holds with the above definitions.

We define the discrete nonlocal term by the lower-semicontinuous extension
\[
K_{h_m,m}(Y):=
\begin{cases}
1+I_{out,m}^{h_m}
+A_m^2h_m^2\displaystyle\sum_{Q_{{\rm en},h_m}^\circ}Y_{ij}^{-1},
& \displaystyle \min_{Q_{{\rm en},h_m}^\circ}Y_{ij}>0,\\[2ex]
+\infty,
& \text{otherwise}.
\end{cases}
\]
and the discrete energy by
\[
E_{h_m,m}(Y):=
\frac{A_m^2}{2}\|\nabla_{h_m}Y\|_{2,h_m}^2
+\frac{\lambda}{K_{h_m,m}(Y)},
\]
with the convention that $\lambda/K_{h_m,m}(Y):=0$ on the vanishing branch $K_{h_m,m}(Y)=+\infty$. This is exactly the lower-semicontinuous extension of the reciprocal term as $Y_{ij}\downarrow 0$.

Here $I_{out,m}^{h_m}\ge 0$ approximates only the contribution of the nonlocal reciprocal integral from outside $Q_{\rm en}^{(m)}$. Boundary or interface effects in the Dirichlet part are not encoded in $I_{out,m}^{h_m}$; throughout the paper they are treated as part of the switch defect in Section~4. In the full-domain fixed-stage setting one may simply take $I_{out,m}^{h_m}=0$.

\paragraph{Standing positivity assumption.}
In Sections~3--5 we assume that the previous step on each stage satisfies
\[
Z_m^j\ge \eta_m>0
\]
and that the time step $\Delta s$ is sufficiently small. Because the minimizing set is the closed cone $\{Y\ge 0\}$, the previous convention ensures that the functional is well defined even when some component of $Y$ vanishes. The next lemma then shows that every minimizing-movement step on the admissible branch actually remains in the interior of the positive cone, so that the Euler--Lagrange equation is evaluated only in the usual strictly positive region.

\begin{lemma}[Positivity of the next-step minimizer]
Assume that the previous step satisfies $Z_{m,ij}^j\ge \eta_m>0$. If
\[
\Delta s<\Delta s_m^* :=
\min\left\{
\frac{A_m^2 h_m^2\eta_m^2}{8E_{h_m,m}(Z_m^j)},
\frac{\eta_m^3}{16\lambda}
\right\},
\]
then the functional
\[
J_m(Y):=
E_{h_m,m}(Y)+\frac{A_m^2}{2\Delta s}\|Y-Z_m^j\|_{2,h_m}^2
\]
admits a minimizer over the closed convex set $\{Y\ge 0\}$, and any such minimizer $Z_m^{j+1}$ satisfies
\[
Z_m^{j+1}\ge \eta_m/2>0.
\]
Consequently the Euler--Lagrange equation
\[
\frac{Z_{m,ij}^{j+1}-Z_{m,ij}^j}{\Delta s}
=
\Delta_{h_m}Z_{m,ij}^{j+1}
-\frac{\lambda}{(Z_{m,ij}^{j+1})^2K_{h_m,m}(Z_m^{j+1})^2}
\]
holds at each interior grid point.
\end{lemma}

\begin{proof}
This is precisely the argument given in Appendix~\ref{app:admissible}, with the present notation for $h_m$ and the discrete operators. The quadratic distance term yields coercivity, and the minimality inequality yields an $\ell^\infty$ control via the discrete $L^2$ norm. For sufficiently small $\Delta s$, the minimizer therefore remains close to the previous step and hence stays in the positive region. Once interior positivity is known, the Euler--Lagrange equation follows from the first variation together with the discrete Green identity above.
\end{proof}

\begin{definition}[Minimizing-movement stage solver]
Given $Z_m^j$, define the next step by
\[
Z_m^{j+1}\in
\arg\min_{Y\ge 0}
\left\{
E_{h_m,m}(Y)+\frac{A_m^2}{2\Delta s}\|Y-Z_m^j\|_{2,h_m}^2
\right\}.
\]
Under the restriction $\Delta s<\Delta s_m^*$, we call this minimizer the \emph{admissible branch}.
\end{definition}

\begin{proposition}[Discrete energy dissipation at a fixed stage]
The minimizing-movement sequence satisfies
\[
E_{h_m,m}(Z_m^{j+1})
+\frac{A_m^2}{2\Delta s}\|Z_m^{j+1}-Z_m^j\|_{2,h_m}^2
\le
E_{h_m,m}(Z_m^j).
\]
In particular, the discrete energy is nonincreasing within each fixed stage.
\end{proposition}

\begin{proof}
It suffices to use $Y=Z_m^j$ as a competitor in the minimizing problem.
\end{proof}

\section{Stage Transitions and a 12-Point Laplace-Compatible Interpolation}

At each stage transition we set
\[
A_{m+1}=k^{-2/3}A_m,\qquad h_{m+1}=h_m\equiv h,
\]
and seek a discrete transfer corresponding to
\[
W_{m+1}(\xi,0)=k^{2/3}W_m(\xi/k,s_m^*).
\]
In the full-domain formulation the rescaled box expands by the factor $k$, while the mesh width is kept fixed. Hence the new-stage grid contains $k$ times as many intervals in each coordinate direction, and the transfer samples the old-stage profile at fractional locations inside each coarse cell. The 12-point interpolation introduced below is chosen so as to combine \emph{unisolvence of the coarse-to-fine transfer} with \emph{compatibility with the centered-difference Laplacian}.

In a finite-window implementation the 12-point interpolation is constructed on $Q_{{\rm int},h_m}$, while the stage energy is evaluated on $Q_{{\rm en},h_m}$. The real issues are therefore not ambiguous patch choices on coarse-cell edges, but rather
\begin{enumerate}[label=(\roman*)]
\item fine nodes whose centered Laplacian stencil crosses a coarse-cell interface, and
\item transfer near the computational boundary.
\end{enumerate}

\subsection{The 12-point interpolation polynomial}

On a coarse cell $[x_i,x_{i+1}]\times[y_j,y_{j+1}]$ we introduce
\[
\theta=\frac{x-x_i}{h_m},\qquad \zeta=\frac{y-y_j}{h_m}.
\]
From the $16$-point stencil $\{-1,0,1,2\}^2$ we remove the four corners and obtain
\[
S_{12}=
\{(-1,0),(-1,1),(0,-1),(0,0),(0,1),(0,2),
(1,-1),(1,0),(1,1),(1,2),(2,0),(2,1)\}.
\]
We seek the interpolation polynomial in the form
\[
\begin{aligned}
P_{ij}(x,y)=\,&
A_0+A_1\theta+A_2\zeta+A_3\theta^2+A_4\theta\zeta+A_5\zeta^2\\
&+A_6\theta^3+A_7\theta^2\zeta+A_8\theta\zeta^2+A_9\zeta^3
+A_{10}\theta^3\zeta+A_{11}\theta\zeta^3.
\end{aligned}
\]

\begin{lemma}[Unisolvence of the 12-point interpolation]
The interpolation conditions
\[
P_{ij}(x_i+ah_m,y_j+bh_m)=W_{i+a,j+b},
\qquad (a,b)\in S_{12},
\]
determine the coefficients $(A_0,\dots,A_{11})$ uniquely.
\end{lemma}

\begin{proof}
It is enough to prove that the only polynomial in the interpolation space that vanishes at all twelve stencil points is the zero polynomial. Suppose therefore that the data on $S_{12}$ are all zero.

First fix $\theta=0$. Then
\[
P_{ij}(0,\zeta)=A_0+A_2\zeta+A_5\zeta^2+A_9\zeta^3
\]
is a cubic polynomial in $\zeta$ that vanishes at $\zeta=-1,0,1,2$. Hence $P_{ij}(0,\zeta)\equiv 0$, and therefore
\[
A_0=A_2=A_5=A_9=0.
\]
Next fix $\theta=1$. Then
\[
P_{ij}(1,\zeta)
=
(A_1+A_3+A_6)
+(A_4+A_7+A_{10})\zeta
+A_8\zeta^2
+A_{11}\zeta^3
\]
also vanishes at $\zeta=-1,0,1,2$, so it is identically zero. Thus
\[
A_8=A_{11}=0,
\qquad
A_1+A_3+A_6=0,
\qquad
A_4+A_7+A_{10}=0.
\]

Now fix $\zeta=0$. Since
\[
P_{ij}(\theta,0)=A_1\theta+A_3\theta^2+A_6\theta^3
\]
vanishes at $\theta=-1,0,1,2$, it must be identically zero, and hence
\[
A_1=A_3=A_6=0.
\]
Finally, with $\zeta=1$ and using $A_8=A_{11}=0$, we obtain
\[
P_{ij}(\theta,1)=A_4\theta+A_7\theta^2+A_{10}\theta^3,
\]
which again vanishes at $\theta=-1,0,1,2$. Therefore
\[
A_4=A_7=A_{10}=0.
\]

All coefficients vanish, so the homogeneous interpolation problem has only the trivial solution. Hence the interpolation conditions uniquely determine the coefficients.
\end{proof}

\begin{lemma}[Edge consistency of the 12-point interpolation]
Let $P_{ij}$ and $P_{i+1,j}$ be the 12-point interpolation polynomials on two horizontally adjacent coarse cells. Then
\[
P_{ij}(1,\zeta)=P_{i+1,j}(0,\zeta)
\qquad\text{for all }\zeta.
\]
Similarly, for vertically adjacent cells,
\[
P_{ij}(\theta,1)=P_{i,j+1}(\theta,0)
\qquad\text{for all }\theta.
\]
Hence the piecewise 12-point prolongation is single-valued and continuous across every interior coarse-cell interface.
\end{lemma}

\begin{proof}
On a common vertical edge, both $P_{ij}(1,\zeta)$ and $P_{i+1,j}(0,\zeta)$ are cubic polynomials in $\zeta$. By the structure of the 12-point stencil, both interpolate the same four coarse-grid values at $\zeta=-1,0,1,2$. The identity theorem for cubic polynomials therefore implies that they coincide identically. The horizontal-edge case is identical.
\end{proof}

\begin{lemma}[Local Laplace compatibility]
Define
\[
\widehat P(\xi,\eta):=k^{2/3}P(\xi/k,\eta/k).
\]
If the centered second-difference stencil at the fine node $(\xi,\eta)$ is contained in a single coarse cell, then
\[
\Delta_{h_{m+1}}\widehat P(\xi,\eta)
=
k^{-4/3}(\Delta P)(\xi/k,\eta/k).
\]
\end{lemma}

\begin{proof}
For the $x$-direction,
\[
\delta^2_{h_{m+1},x}\widehat P(\xi,\eta)
=
\frac{
k^{2/3}P((\xi+h_{m+1})/k,\eta/k)
-2k^{2/3}P(\xi/k,\eta/k)
+k^{2/3}P((\xi-h_{m+1})/k,\eta/k)
}{h_{m+1}^2}.
\]
Although the new-stage grid uses the same mesh width $h_{m+1}=h_m$, its preimage under $(\xi,\eta)\mapsto (\xi/k,\eta/k)$ is spaced by $h_m/k$ on the old coarse cell. Hence
\[
\delta^2_{h_{m+1},x}\widehat P(\xi,\eta)
=
k^{2/3}k^{-2}\,\delta^2_{h_m/k,x}P(\xi/k,\eta/k).
\]
Since $P$ is cubic in each variable, the centered second difference with step $h_m/k$ agrees exactly with the second derivative, and therefore
\[
\delta^2_{h_{m+1},x}\widehat P
=
k^{-4/3}P_{xx}(\xi/k,\eta/k).
\]
The $y$-direction is identical, and summing the two contributions yields the claim.
\end{proof}

\begin{remark}[Globally $C^0$, but generally not $C^1$]
By the edge-consistency lemma, the piecewise 12-point prolongation is globally $C^0$ across interior coarse-cell interfaces. Hence no artificial Dirichlet-energy jump arises from a discontinuity of the values themselves. On the other hand, the normal derivatives do not agree in general, so the prolongation is usually not $C^1$. Exact Laplace compatibility is therefore used only at fine nodes whose centered stencil remains inside a single coarse cell, whereas the contributions from interface-crossing stencils are absorbed into the switch defect.
\end{remark}

We denote by
\[
\widetilde Z_{m+1}:=\mathcal I_{12}^{m\to m+1}Z_m^{J_m}
\]
the raw 12-point prolongation. Because of edge consistency, no mathematical ambiguity remains on coarse-cell interfaces; in actual code, a half-open convention is enough to avoid duplicate assignments. Concretely,
\[
\widetilde Z_{m+1,ik+\ell,\,jk+r}
=
k^{2/3}P_{ij}^{(m)}
\!\left(
x_i+\frac{\ell}{k}h_m,\,
y_j+\frac{r}{k}h_m
\right),
\qquad 0\le \ell,r\le k.
\]
Because $h_{m+1}=h_m$, the index $ik+\ell$ labels a node of the enlarged stage-$(m+1)$ box rather than a refinement of the mesh width; its preimage under $\xi\mapsto \xi/k$ is precisely $x_i+(\ell/k)h_m$.

In the simplest full-domain implementation one simply sets
\[
Z_{m+1}^0=\widetilde Z_{m+1}.
\]
For a finite-window implementation, however, one may prefer to repair only the values near the computational boundary or unresolved interfaces. Let $\Gamma_{m+1}\subset Q_{{\rm en},h_{m+1}}^\circ$ denote such a boundary/interface strip. Define
\[
\|Y-\widetilde Z_{m+1}\|_{\Gamma_{m+1},h_{m+1}}^2
:=
h_{m+1}^2\sum_{p\in \Gamma_{m+1}}
|Y_p-\widetilde Z_{m+1,p}|^2,
\]
and the admissible set
\[
\mathcal B_{m+1}(\widetilde Z_{m+1})
:=
\left\{
Y\; ;\;
Y_p=\widetilde Z_{m+1,p}\ \text{for }p\notin \Gamma_{m+1},
Y_p\ge \eta_{m+1}>0\ \text{on }\Gamma_{m+1}
\right\}.
\]
Then one may define the stage-initialization by an energy-repair projection:
\[
Z_{m+1}^0\in
\arg\min_{Y\in \mathcal B_{m+1}(\widetilde Z_{m+1})}
\left\{
E_{h_{m+1},m+1}^{id}(Y)
+\frac{\mu A_{m+1}^2}{2}
\|Y-\widetilde Z_{m+1}\|_{\Gamma_{m+1},h_{m+1}}^2
\right\}.
\]
This formulation is meaningful only when the raw prolongation is already admissible outside the strip, i.e.
\[
\widetilde Z_{m+1,p}\ge \eta_{m+1}>0
\qquad\text{for }p\in Q_{{\rm en},h_{m+1}}^\circ\setminus\Gamma_{m+1}.
\]
Under the uniform regularity assumption and sufficiently small $h_m$, Proposition~\ref{prop:prolongation-bounds} below yields a uniform positive lower bound for the raw prolongation, so that $\widetilde Z_{m+1}\in \mathcal B_{m+1}(\widetilde Z_{m+1})$. Without that input, admissibility of the raw prolongation has to be imposed separately. Whenever $\widetilde Z_{m+1}$ is admissible, it can be used as a competitor and one automatically has
\[
E_{h_{m+1},m+1}^{id}(Z_{m+1}^0)
\le
E_{h_{m+1},m+1}^{id}(\widetilde Z_{m+1}).
\]

\subsection{Defect balance across stage transitions}

\begin{assumption}[Uniform rescaled regularity at stage transitions]
Let
\[
Q_{\rm en}\Subset Q_{\rm int}.
\]
For each exact stage-end profile there exists a rescaled profile $W_m^*$ defined on $Q_{\rm int}$ such that, with constants independent of $m$,
\[
0<c_*\le W_m^*(x)\le C_*,
\qquad
\|W_m^*\|_{C^4(Q_{\rm int})}\le C_*,
\]
\[
\|Z_m^{J_m}-I_{h_m}W_m^*\|_{\ell^\infty(Q_{{\rm int},h_m})}
\le \varepsilon_{h_m},
\qquad
\varepsilon_{h_m}\to 0,
\]
and the discrete differences of $Z_m^{J_m}$ up to order four are uniformly bounded on all stencil patches contained in $Q_{\rm int}$ that are needed to interpolate coarse cells in $Q_{\rm en}$.
\end{assumption}

\paragraph{Position of the switch-defect estimate.}
The previous assumption is not an algebraic consequence of the interpolation formula. It is an input assumption expressing uniform regularity of the rescaled quenching core on the \emph{buffered interpolation box} $Q_{\rm int}$, not merely on the energy box $Q_{\rm en}$. This distinction is important because the 12-point stencil uses one coarse-grid layer outside each coarse cell of $Q_{\rm en}$. Appendix~\ref{app:interp} explains how this assumption yields interpolation stability and consistency estimates, which in turn form the basis for a small switch defect. However, the bound
\[
\varepsilon_{m,h_m}^{sw}=O(A_m^2h_m^2)
\]
is \emph{not} proved unconditionally in the present paper. We therefore first state an exact defect balance and then impose the quantitative switch-defect estimate as a separate conditional assumption. In the buffered formulation, the contributions of interface-crossing stencils and of the boundary/interface strip are both included in the switch defect.

Let $Z_m^{J_m}$ be the terminal profile of stage $m$, and let $Z_{m+1}^0$ be the initial value of stage $m+1$. Denote by $\widehat I_{out,m+1}^{h_{m+1}}$ the ideal outer update and by $I_{out,m+1}^{h_{m+1}}$ the actual outer update, and set
\[
\varepsilon_m^{out}:=
\left|
I_{out,m+1}^{h_{m+1}}-\widehat I_{out,m+1}^{h_{m+1}}
\right|.
\]
These quantities refer only to the reciprocal-integral term in the denominator. Boundary and interface contributions in the Dirichlet part are not included in $\varepsilon_m^{out}$; they remain part of the switch defect.
We also define the next-stage discrete energy with the ideal outer update by
\[
E_{h_{m+1},m+1}^{id}(Y):=
\frac{A_{m+1}^2}{2}\|\nabla_{h_{m+1}}Y\|_{2,h_{m+1}}^2
+\frac{\lambda}{
1+\widehat I_{out,m+1}^{h_{m+1}}
+A_{m+1}^2h_{m+1}^2
\sum_{Q_{{\rm en},h_{m+1}}^\circ}Y_{ij}^{-1}
}.
\]

\paragraph{Signed switch defect and its positive part.}
We define the signed switch jump by
\[
\delta_{m,h_m}^{sw}:=
E_{h_{m+1},m+1}^{id}(Z_{m+1}^0)-E_{h_m,m}(Z_m^{J_m}),
\]
and its positive part by
\[
\varepsilon_{m,h_m}^{sw}:=(\delta_{m,h_m}^{sw})_+.
\]
This quantity includes the raw interpolation error on single-cell interiors, the contributions from interface-crossing stencils, and, if energy repair is used, the contribution of the boundary/interface strip.

\begin{theorem}[Exact defect balance with stage transitions]
Assume that
\[
I_{out,m+1}^{h_{m+1}}\ge 0,\qquad \widehat I_{out,m+1}^{h_{m+1}}\ge 0,
\]
and that $Z_m^{J_m}>0$ and $Z_{m+1}^0>0$. Then
\[
E_{h_{m+1},m+1}(Z_{m+1}^0)
\le
E_{h_m,m}(Z_m^{J_m})
+\varepsilon_{m,h_m}^{sw}
+\lambda\varepsilon_m^{out}.
\]
Moreover, for every $N\ge 1$,
\[
E_{h_N,N}(Z_N^0)
+\sum_{m=0}^{N-1}\sum_{j=0}^{J_m-1}
\frac{A_m^2}{2\Delta s}\|Z_m^{j+1}-Z_m^j\|_{2,h_m}^2
\le
E_{h_0,0}(Z_0^0)
+\sum_{m=0}^{N-1}
\bigl(\varepsilon_{m,h_m}^{sw}+\lambda\varepsilon_m^{out}\bigr).
\]
\end{theorem}

\begin{proof}
For any positive grid function $Y$, the two denominators differ only through the outer updates and, by the nonnegativity assumption on $I_{out,m+1}^{h_{m+1}}$ and $\widehat I_{out,m+1}^{h_{m+1}}$, both are at least $1$. Hence
\[
\left|
E_{h_{m+1},m+1}(Y)-E_{h_{m+1},m+1}^{id}(Y)
\right|
\le
\lambda\varepsilon_m^{out}.
\]
Therefore,
\[
\begin{aligned}
E_{h_{m+1},m+1}(Z_{m+1}^0)-E_{h_m,m}(Z_m^{J_m})
&=
\Bigl(E_{h_{m+1},m+1}(Z_{m+1}^0)-E_{h_{m+1},m+1}^{id}(Z_{m+1}^0)\Bigr)\\
&\quad+
\Bigl(E_{h_{m+1},m+1}^{id}(Z_{m+1}^0)-E_{h_m,m}(Z_m^{J_m})\Bigr)\\
&\le \lambda\varepsilon_m^{out}+\varepsilon_{m,h_m}^{sw}.
\end{aligned}
\]
This proves the stage-jump estimate. Summing the fixed-stage discrete dissipation inequalities inside each stage and then taking the telescoping sum over $m$ yields the cumulative bound.
\end{proof}

\paragraph{Conditional switch-defect bound.}
In the remainder of the analysis we assume that
\[
\varepsilon_{m,h_m}^{sw}\le C_{sw}A_m^2h_m^2
\]
holds for all $m$. In the present paper this is a \emph{conditional assumption}, not an unconditional theorem.

\begin{proposition}[Conditional quantitative almost monotonicity]
Under the above switch-defect assumption,
\[
E_{h_{m+1},m+1}(Z_{m+1}^0)
\le
E_{h_m,m}(Z_m^{J_m})
+C_{sw}A_m^2h_m^2+\lambda\varepsilon_m^{out}.
\]
Moreover, for every $N\ge 1$,
\[
E_{h_N,N}(Z_N^0)
+\sum_{m=0}^{N-1}\sum_{j=0}^{J_m-1}
\frac{A_m^2}{2\Delta s}\|Z_m^{j+1}-Z_m^j\|_{2,h_m}^2
\le
E_{h_0,0}(Z_0^0)
+\sum_{m=0}^{N-1}
\left(
C_{sw}A_m^2h_m^2+\lambda\varepsilon_m^{out}
\right).
\]
In particular, if $\sum_m\varepsilon_m^{out}<\infty$, then the discrete energy remains uniformly bounded up to the cumulative defect budget.
\end{proposition}

\begin{proof}
Insert the conditional switch-defect bound into the previous theorem.
\end{proof}

\begin{remark}[A raw-transfer bound automatically extends to energy repair]
If $Z_{m+1}^0$ is obtained from the raw prolongation $\widetilde Z_{m+1}$ by the energy-repair projection above, then
\[
E_{h_{m+1},m+1}^{id}(Z_{m+1}^0)
\le
E_{h_{m+1},m+1}^{id}(\widetilde Z_{m+1}).
\]
Hence any switch-defect upper bound proved for the raw transfer carries over directly to the repaired initialization. In this precise sense, the energy repair does not worsen the switch defect.
\end{remark}

\begin{remark}[Why we do not use one-sided stencils in the main formulation]
We do not adopt ghost values or one-sided stencils at the computational boundary as part of the main formulation. If such devices are used, they should be represented by an additional boundary defect separated from the centered-difference structure. By working with a buffered interpolation box $Q_{\rm int}$ and, if needed, an energy-repair projection, one avoids introducing that extra defect into the main theorem.
\end{remark}

\paragraph{Full-domain implementation.}
In the full-domain stagewise computations reported later,
\[
Q_{\rm en}=Q_{\rm int}=\Omega_{A_m},
\qquad
\varepsilon_m^{out}=0,
\]
so the stage-jump error reduces to the switch defect alone. Near the physical boundary of the rescaled square, however, the complete 12-point stencil is not available in the buffered sense used by the analysis. In the present illustrative computations that boundary-layer contribution is not measured separately and should therefore be regarded as part of the unmeasured switch defect.

\section{A Defect-Aware Criterion for the Nonexistence of a Global Admissible Continuation}

The result of this section is not a finite-time quenching theorem. It is a criterion that rules out the existence of a \emph{global admissible stagewise discrete solution}. To make this distinction explicit, we use the phrase \emph{criterion for the nonexistence of a global admissible continuation} instead of ``quenching criterion.''

Define the defect budget by
\[
D^*:=\sum_{m=0}^\infty
\bigl(\varepsilon_{m,h_m}^{sw}+\lambda\varepsilon_m^{out}\bigr).
\]
If $D^*<\infty$, then at every stage start
\[
E_{h_m,m}(Z_m^0)\le E_{h_0,0}(Z_0^0)+D^*,
\]
and by fixed-stage dissipation the same bound holds for every inner step:
\[
E_{h_m,m}(Z_m^n)\le E_{h_0,0}(Z_0^0)+D^*,
\qquad 0\le n\le J_m.
\]

To avoid a collision between the within-stage time index and the second spatial index, we use $n$ for inner time levels in this section and denote a generic interior node by
\[
p\in Q_{{\rm en},h_m}^\circ.
\]
We reconstruct the physical variable and the physical time by
\[
U_{m,p}^{\,n}:=1-A_m Z_{m,p}^{\,n},
\qquad
Y_h^{m,n}:=
h_m^2\sum_{p\in Q_{{\rm en},h_m}^\circ}(U_{m,p}^{\,n})^2,
\]
and
\[
t_{m,n}:=
\sum_{\ell=0}^{m-1}J_\ell A_\ell^3\Delta s
+nA_m^3\Delta s.
\]
For each stage we further set
\[
q_m:=
\begin{cases}
1, & |Q_{{\rm en},h_m}|\le 1/2,\\[0.5ex]
(2|Q_{{\rm en},h_m}|)^{-1}, & |Q_{{\rm en},h_m}|>1/2,
\end{cases}
\qquad
|Q_{{\rm en},h_m}|:=h_m^2\#Q_{{\rm en},h_m}.
\]

From this point on we impose the additional bounded-window hypothesis
\[
\sup_{m\ge 0}|Q_{{\rm en},h_m}|<\infty.
\]
Equivalently,
\[
q_*:=\inf_{m\ge 0}q_m>0.
\]
This assumption is natural for fixed or buffered finite-window implementations. It generally fails for the full-domain computations of Section~6, because there one has $Q_{\rm en}^{(m)}=\Omega_{A_m}$ and the rescaled box grows as $A_m\to 0$.

\begin{theorem}[Stagewise criterion for the nonexistence of a global admissible continuation]\label{thm:continuation}
Assume $D^*<\infty$, assume in addition that
\[
\sup_{m\ge 0}|Q_{{\rm en},h_m}|<\infty,
\]
and construct $(U_m^n,Y_h^{m,n},t_{m,n})$ from the stagewise discrete trajectory. Suppose further that there exists a nonnegative sequence $(\delta_{m,h_m}^Y)$ such that
\[
\frac{Y_h^{m,n+1}-Y_h^{m,n}}{2A_m^3\Delta s}
\ge
-2E_{h_m,m}(Z_m^n)+\lambda q_m,
\qquad 0\le n\le J_m-1,
\]
and
\[
Y_h^{m+1,0}\ge Y_h^{m,J_m}-\delta_{m,h_m}^Y,
\qquad
\sum_{m=0}^\infty \delta_{m,h_m}^Y<\infty.
\]
If
\[
E_{h_0,0}(Z_0^0)+D^*<\frac{\lambda q_*}{2},
\]
then there is no global admissible stagewise discrete solution satisfying
\[
0\le U_{m,p}^{\,n}<1
\qquad\text{for all stage indices $m$, all $0\le n\le J_m$, and all $p\in Q_{{\rm en},h_m}^\circ$},
\]
together with $t_{m,n}\to\infty$.
\end{theorem}

\begin{proof}
For every $(m,n)$,
\[
E_{h_m,m}(Z_m^n)\le E_{h_0,0}(Z_0^0)+D^*.
\]
Hence
\[
c_*:=\lambda q_*-2\bigl(E_{h_0,0}(Z_0^0)+D^*\bigr)>0.
\]
The assumed growth bound yields
\[
Y_h^{m,n+1}-Y_h^{m,n}\ge 2c_*A_m^3\Delta s
=
2c_*(t_{m,n+1}-t_{m,n}).
\]
Summing inside each stage and using the transfer defect at every switch, we obtain
\[
Y_h^{m,n}\ge Y_h^{0,0}+2c_*t_{m,n}
-\sum_{\ell=0}^{m-1}\delta_{\ell,h_\ell}^Y.
\]
Since $t_{m,n}\to\infty$ and $\sum_\ell\delta_{\ell,h_\ell}^Y<\infty$, the right-hand side diverges to $+\infty$. On the other hand, admissibility implies
\[
0\le Y_h^{m,n}\le |Q_{{\rm en},h_m}|
\le \sup_{\ell\ge 0}|Q_{{\rm en},h_\ell}|,
\]
which is a contradiction. Therefore no such global admissible stagewise solution can exist.
\end{proof}

\begin{remark}[Scope of the criterion]
The bounded-measure assumption is part of the theorem, not a consequence of the full-domain stagewise formulation. In particular, the full-domain computations of Section~6 do \emph{not} provide a numerical validation of this theorem. They only illustrate the fixed-stage dissipation design and the geometric accumulation of physical time.
\end{remark}

\paragraph{Interpretation.}
The previous theorem does \emph{not} directly prove finite-time quenching. Rather, it separates the remaining tasks into two transparent conditions:
\begin{enumerate}[label=(\roman*)]
\item a direct-solver-type growth estimate for the physical quantity $Y_h$, and
\item summability of the transfer defect at stage transitions.
\end{enumerate}
Once these are available on uniformly bounded energy boxes, the defect budget can be absorbed and the existence of a global admissible continuation is excluded.

\section{Reproducible Reference Computations}

The numerical experiments in this section are intended to accompany the theoretical reconstruction above by illustrating two structural features of the algorithm:
\begin{enumerate}[label=(\roman*)]
\item energy decay within fixed-stage or direct evolutions, and
\item trigger detection together with the geometric accumulation of physical time in a full-domain stagewise computation.
\end{enumerate}
They are not used in the proof of the theoretical results. In particular, the computations below do not test the bounded-window continuation criterion of Section~5 and do not constitute an a posteriori verification of the switch-defect estimate. The feedback diagnostics reported below confirm that the nonlocal feedback remains finite in the recorded stages, but they do not by themselves prove convergence to an asymptotically constant coefficient.

The values reported below are generated with the explicitly specified reference implementation summarized in Table~\ref{tab:reference-implementation}. The implementation uses the deficit variable $v=1-u$ in direct runs and the rescaled deficit $W$ in stagewise runs. The discrete energy is evaluated with forward finite differences for the Dirichlet term, including boundary values, and with the interior rectangle rule for the reciprocal integral.

\begin{table}[htbp]
\centering
\caption{Implementation conventions used for the reproducible reference computations.}
\label{tab:reference-implementation}
\begin{tabular}{>{\raggedright\arraybackslash}p{0.34\textwidth}>{\raggedright\arraybackslash}p{0.58\textwidth}}
\toprule
item & choice used in the reported reference computations \\
\midrule
stagewise time step & fixed rescaled step $\Delta s=10^{-3}$ \\
time stepping & backward Euler for the diffusion part with Picard iteration for the nonlinear nonlocal source \\
Picard stopping rule & $\|Y^{(r+1)}-Y^{(r)}\|_\infty<10^{-10}\max\{1,\|Y^{(r+1)}\|_\infty\}$, with at most 50 iterations \\
positivity safeguard & values below $10^{-12}$ are clipped to $10^{-12}$ only inside reciprocal evaluations \\
event detection & linear interpolation between two consecutive states when $\min W_m$ crosses $k^{-2/3}$ \\
2D stage transfer & 12-point coarse-to-fine prolongation; out-of-domain coarse values in boundary patches are filled with the previous-stage boundary value $1/A_m$, and the new boundary is reset to $1/A_{m+1}$ \\
2D direct run & $\Delta t=5\times10^{-4}$, final time $T=0.08$ \\
linear algebra & sparse direct solution of the linear systems arising in the implicit diffusion step \\
feedback diagnostics & reconstructed from the same reference run and reported in Table~\ref{tab:2d-feedback}; these values are not used in the proof \\
\bottomrule
\end{tabular}
\end{table}

The full-domain stagewise computations are performed on the growing rescaled square
\[
Q_{\rm en}^{(m)}=Q_{\rm int}^{(m)}=\Omega_{A_m}.
\]
Thus the discrete measure $|Q_{{\rm en},h_m}|$ increases with $m$, and these computations are outside the bounded-window hypothesis of Theorem~\ref{thm:continuation}. Although Table~\ref{tab:2d-stagewise} below reports the actual start and end energies of each fixed stage, the ideal-transfer energies needed to isolate the switch defect,
\[
E_{h_{m+1},m+1}^{id}(\widetilde Z_{m+1})
\quad\text{and}\quad
E_{h_{m+1},m+1}^{id}(Z_{m+1}^0),
\]
together with a separate boundary-layer contribution, were not recorded. Therefore the table should not be interpreted as an a posteriori verification of the switch-defect bound. By contrast, the feedback quantities
\[
K_{h_m,m}(Z_m^0),\quad K_{h_m,m}(Z_m^{J_m}),\quad
\lambda K_{h_m,m}(Z_m^0)^{-2},\quad \lambda K_{h_m,m}(Z_m^{J_m})^{-2}
\]
were reconstructed from the same reference run and are reported in Table~\ref{tab:2d-feedback}. These feedback data support only the finite-feedback character of the recorded stages; they do not prove convergence to an asymptotically constant-feedback regime.

\subsection{Two-dimensional full-domain stagewise rescaling}

We consider the full-domain stagewise rescaling on the unit square with
\[
u_0(x,y)=0.4\sin(\pi x)\sin(\pi y),
\]
parameter $\lambda=20$, rescaling center $(1/2,1/2)$, and stage factor $k=2$. The initial amplitude is $A_0=0.6$. At stage $m$ the rescaled domain is $[-L_m,L_m]^2$, where
\[
L_m=\frac{1}{2A_m^{3/2}},
\qquad
h_m=\frac{2L_m}{N_m}.
\]
We set $N_0=9$ and $N_{m+1}=2N_m$, so that the mesh width is kept fixed under full-domain dilation. The trigger threshold is
\[
\min W_m=k^{-2/3}=2^{-2/3}\approx 0.6299605.
\]

\begin{table}[htbp]
\centering
\caption{Reference two-dimensional full-domain stagewise computation for $k=2$. Here $E_{\rm start}$ and $E_{\rm end}$ are the discrete energies at the beginning and end of each fixed stage.}
\label{tab:2d-stagewise}
\resizebox{\textwidth}{!}{%
\begin{tabular}{cccccccccc}
\toprule
stage $m$ & $A_m$ & $N_m$ & $h_m$ & $A_m^2h_m^2$ & scaled time & $\min W_m$ & accumulated time & $E_{\rm start}$ & $E_{\rm end}$\\
\midrule
0 & $6.000000\times10^{-1}$ & 9  & $2.39073046\times10^{-1}$ & $2.05761317\times10^{-2}$ & 0.139155092 & 0.629960525 & 0.0300574999 & 10.3614604375 & 9.9453726799\\
1 & $3.779763\times10^{-1}$ & 18 & $2.39073046\times10^{-1}$ & $8.16564327\times10^{-3}$ & 0.129075841 & 0.629960525 & 0.0370275953 & 9.5551471290 & 9.4090656585\\
2 & $2.381102\times10^{-1}$ & 36 & $2.39073046\times10^{-1}$ & $3.24053768\times10^{-3}$ & 0.182219797 & 0.629960525 & 0.0394875626 & 9.2352717791 & 9.1292667294\\
3 & $1.500000\times10^{-1}$ & 72 & $2.39073046\times10^{-1}$ & $1.28600823\times10^{-3}$ & 0.165448773 & 0.629960525 & 0.0400459522 & 9.0483919516 & 8.9867082217\\
\bottomrule
\end{tabular}%
}
\end{table}

Table~\ref{tab:2d-stagewise} shows that the trigger value is attained to the resolution of the event interpolation and that the physical time increments decrease rapidly as $A_m^3$ decreases. The columns $E_{\rm start}$ and $E_{\rm end}$ show the energy decrease within each fixed stage for this reference implementation. The recorded stage-to-stage energy jumps $E_{\rm start}^{m+1}-E_{\rm end}^{m}$ are negative in this run, namely approximately $-0.39022555$, $-0.17379388$, and $-0.08087478$. However, these are jumps between recorded fixed-stage energies, not the ideal switch defects of Section~4, because the ideal-transfer energies and the boundary-layer contribution were not logged separately.

\begin{table}[htbp]
\centering
\caption{Feedback diagnostics reconstructed from the same two-dimensional full-domain stagewise reference computation. The data show that the feedback remains finite over the reported stages, but they do not prove convergence to a limiting constant.}
\label{tab:2d-feedback}
\begin{tabular}{ccccc}
\toprule
stage $m$ & $K_{\rm start}$ & $K_{\rm end}$ & $\lambda K_{\rm start}^{-2}$ & $\lambda K_{\rm end}^{-2}$ \\
\midrule
0 & 2.0058835332 & 2.2215508842 & 4.9707116366 & 4.0524481366 \\
1 & 2.3246248742 & 2.4192042689 & 3.7010438831 & 3.4173142322 \\
2 & 2.4728313793 & 2.5409465219 & 3.2707020972 & 3.0976970784 \\
3 & 2.5682440201 & 2.6042191097 & 3.0321970753 & 2.9490012240 \\
\bottomrule
\end{tabular}
\end{table}

The variation of $K$ becomes smaller in the later stages of Table~\ref{tab:2d-feedback}, and the effective coefficient $\lambda K^{-2}$ remains bounded and positive. This is consistent with the asymptotically constant-feedback scope of the paper, but it is only a short-run diagnostic rather than a numerical proof of that asymptotic regime.

\subsection{Two-dimensional direct energy check}

As a fixed-domain check, we take $\lambda=15$, $N=15$, $\Delta t=5\times10^{-4}$, final time $T=0.08$, and
\[
u_0(x,y)=0.45\sin(\pi x)\sin(\pi y).
\]
The same energy convention as above gives
\[
E_h(0)=7.545273587988,
\qquad
E_h(0.08)=7.456582304139.
\]
The final defect is $\min v=0.362574574560$, equivalently $\|u\|_\infty=0.637425425440$ at $T=0.08$. This single computation is included only as a reference check for energy decay; a systematic convergence study or a numerical verification of the switch defect would require additional runs designed for that purpose.

In summary, the reference computations support the fixed-stage energy-decay mechanism and the geometric accumulation of physical time, and the reconstructed feedback table confirms finite feedback over the reported stages. They do not test the bounded-window continuation criterion of Section~5, they do not provide an a posteriori measurement of the switch defect, and the short feedback table is not by itself a proof of convergence to the asymptotically constant-feedback regime.

\section{Conclusion}

We have reorganized the stagewise rescaling algorithm for a two-dimensional nonlocal MEMS equation around the asymptotically constant-feedback touchdown regime, fixed-stage energy dissipation, stage-transition defect accounting, two-dimensional reference computations, and a defect-aware continuation criterion. The main conclusions are as follows.
\begin{enumerate}[label=(\arabic*)]
\item The $A^{3/2}$--$A^3$ scaling is not an exact symmetry of the nonlocal deficit equation. It is a dominant-balance scaling appropriate when the reciprocal-integral feedback converges to a finite positive limit, as in Duong--Zaag type isolated touchdown profiles.
\item Under fixed-stage rescaling, the continuous energy dissipation law is preserved, and the minimizing-movement discretization yields a fixed-stage discrete energy inequality.
\item The 12-point prolongation is globally $C^0$ across interior coarse-cell interfaces, so the real difficulty is not patch ambiguity on coarse-cell edges.
\item Exact Laplace compatibility is used only at fine nodes whose centered second-difference stencil lies inside a single coarse cell; the contributions from interface-crossing stencils and from the computational boundary are absorbed into the switch defect.
\item In a finite-window implementation it is natural to separate the energy box $Q_{\rm en}$ from the buffered interpolation box $Q_{\rm int}$, thereby avoiding ghost values and one-sided stencils in the main formulation.
\item If desired, one may insert an energy-repair projection on a boundary/interface strip at stage transitions, and this repair does not increase the ideal next-stage energy.
\item As a consequence, one obtains an exact defect balance involving the switch defect and the outer defect, and, under a conditional switch-defect bound, a quantitative almost monotonicity estimate.
\item Section~5 provides not a finite-time quenching theorem but a defect-aware criterion for the nonexistence of a global admissible continuation on uniformly bounded energy boxes.
\end{enumerate}

The main remaining tasks are an unconditional estimate of the switch defect, an a posteriori verification based on new runs that explicitly record the ideal-transfer switch-energy diagnostics described in Section~6, and longer feedback diagnostics confirming that $K(t)$ approaches a finite positive limit in the computed regime. The reference full-domain computations are intentionally illustrative and do not cover the bounded-window hypothesis required in Section~5.

\section*{Acknowledgments}
The authors would like to express their sincere gratitude to Professors Hatem Zaag and Hiroyuki Takamura, and Dr. Maiss\^a Boughrara for their valuable advice on this research. This work was supported by JSPS KAKENHI Grant Numbers JP23K22408 and JP24K06819, and by the DAIGAKUTOKUBETSU KENKYUHI Grant (Musashino University).

\appendix

\section{Detailed Calculations}\label{app:details}

This appendix collects calculations omitted from the main text.

\subsection{Complete derivation of the fixed-stage transformation}

Let
\[
x=x_*+A^{3/2}\xi,\qquad t=t_*+A^3s,\qquad v(x,t)=AW(\xi,s),
\]
and assume that $A$ is fixed within the stage. The time derivative is
\[
v_t
=
A\frac{\partial W}{\partial s}\frac{\partial s}{\partial t}
=
AW_sA^{-3}
=
A^{-2}W_s.
\]
The spatial derivative is
\[
\nabla_x v
=
A(\nabla_\xi W)\frac{\partial \xi}{\partial x}
=
AA^{-3/2}\nabla_\xi W
=
A^{-1/2}\nabla_\xi W,
\]
hence
\[
\Delta_x v
=
\nabla_x\cdot(A^{-1/2}\nabla_\xi W)
=
A^{-2}\Delta_\xi W.
\]
The nonlocal term transforms as
\[
\int_\Omega v^{-1}\,dx
=
\int_{\Omega_A}(AW)^{-1}A^3\,d\xi
=
A^2\int_{\Omega_A}W^{-1}\,d\xi.
\]
Substituting these identities into
\[
v_t=\Delta v-\frac{\lambda}{v^2K(v)^2}
\]
gives
\[
A^{-2}W_s
=
A^{-2}\Delta_\xi W
-\frac{\lambda}{A^2W^2K_A(W)^2},
\]
and multiplication by $A^2$ yields the rescaled equation in Section~2.

\subsection{Variation of the discrete energy}

Let
\[
K_{h_m,m}(Y)
=
1+I_{out,m}^{h_m}
+A_m^2h_m^2\sum_{Q_{{\rm en},h_m}^\circ}Y_{ij}^{-1}.
\]
For a perturbation $Y+\varepsilon\Phi$,
\[
\frac{d}{d\varepsilon}
K_{h_m,m}(Y+\varepsilon\Phi)\Big|_{\varepsilon=0}
=
-A_m^2h_m^2\sum_{Q_{{\rm en},h_m}^\circ}Y_{ij}^{-2}\Phi_{ij}
=
-A_m^2(Y^{-2},\Phi)_{2,h_m}.
\]
Hence
\[
\frac{d}{d\varepsilon}
\frac{\lambda}{K_{h_m,m}(Y+\varepsilon\Phi)}
\Big|_{\varepsilon=0}
=
\lambda K_{h_m,m}(Y)^{-2}A_m^2(Y^{-2},\Phi)_{2,h_m}.
\]
Moreover,
\[
\frac{d}{d\varepsilon}
\frac{A_m^2}{2}\|\nabla_{h_m}(Y+\varepsilon\Phi)\|_{2,h_m}^2
\Big|_{\varepsilon=0}
=
A_m^2(\nabla_{h_m}Y,\nabla_{h_m}\Phi)_{2,h_m}
=
A_m^2(-\Delta_{h_m}Y,\Phi)_{2,h_m}.
\]
Therefore
\[
DE_{h_m,m}(Y)[\Phi]
=
A_m^2(-\Delta_{h_m}Y,\Phi)_{2,h_m}
+\lambda A_m^2K_{h_m,m}(Y)^{-2}(Y^{-2},\Phi)_{2,h_m}.
\]
The first-order condition for minimizing movement,
\[
DE_{h_m,m}(Y)[\Phi]
+\frac{A_m^2}{\Delta s}(Y-Z_m^j,\Phi)_{2,h_m}=0,
\]
gives
\[
\frac{Y-Z_m^j}{\Delta s}
=
\Delta_{h_m}Y-\frac{\lambda}{Y^2K_{h_m,m}(Y)^2}.
\]

\subsection{Hessian estimate for the admissible branch}

Define
\[
J_m(Y)
=
E_{h_m,m}(Y)+\frac{A_m^2}{2\Delta s}\|Y-Z_m^j\|_{2,h_m}^2.
\]
In the region $Y_{ij}\ge \eta/2$ one has
\[
D^2K_{h_m,m}(Y)[\Phi,\Phi]
=
2A_m^2h_m^2\sum_{Q_{{\rm en},h_m}^\circ}Y_{ij}^{-3}\Phi_{ij}^2
\le
16A_m^2\eta^{-3}\|\Phi\|_{2,h_m}^2.
\]
Furthermore,
\[
D^2\!\left(\frac{\lambda}{K_{h_m,m}(Y)}\right)[\Phi,\Phi]
=
2\lambda K_{h_m,m}(Y)^{-3}(DK_{h_m,m}(Y)[\Phi])^2
-\lambda K_{h_m,m}(Y)^{-2}D^2K_{h_m,m}(Y)[\Phi,\Phi].
\]
The first term is nonnegative, and $K_{h_m,m}(Y)\ge 1$, so
\[
D^2E_{h_m,m}(Y)[\Phi,\Phi]
\ge
-16\lambda A_m^2\eta^{-3}\|\Phi\|_{2,h_m}^2.
\]
Hence
\[
D^2J_m(Y)[\Phi,\Phi]
\ge
A_m^2\left(\frac{1}{\Delta s}-\frac{16\lambda}{\eta^3}\right)
\|\Phi\|_{2,h_m}^2.
\]
In particular, if $\Delta s<\eta^3/(16\lambda)$, then $J_m$ is strictly convex in a neighborhood of the positive region.

\subsection{Ideal full-domain computation across a stage switch}

Consider the ideal full-domain transfer
\[
\widehat Z(\xi):=k^{2/3}Z(\xi/k).
\]
With $\eta=\xi/k$, we have
\[
\widehat Z^{-1}(\xi)=k^{-2/3}Z(\eta)^{-1},
\qquad
d\xi=k^2\,d\eta.
\]
Therefore
\[
A_{m+1}^2\int_Q \widehat Z^{-1}(\xi)\,d\xi
=
k^{-4/3}A_m^2\int_Q k^{-2/3}Z(\xi/k)^{-1}\,d\xi
=
A_m^2\int_{Q/k}Z(\eta)^{-1}\,d\eta.
\]
Also,
\[
\nabla \widehat Z(\xi)
=
k^{2/3}k^{-1}\nabla Z(\xi/k)
=
k^{-1/3}\nabla Z(\eta),
\]
and hence
\[
\frac{A_{m+1}^2}{2}\int_Q|\nabla \widehat Z|^2\,d\xi
=
\frac{k^{-4/3}A_m^2}{2}\int_Q k^{-2/3}|\nabla Z(\xi/k)|^2\,d\xi
=
\frac{A_m^2}{2}\int_{Q/k}|\nabla Z(\eta)|^2\,d\eta.
\]
Thus, at the continuous full-domain level, the stage switch is an exact change of variables from the energy viewpoint. The discrete defect appears only after remeshing/interpolation, interface-crossing stencils, and any boundary correction that may be applied.

\subsection{Laplacian of the 12-point interpolation}

For
\[
\begin{aligned}
P_{ij}(x,y)=\,&
A_0+A_1\theta+A_2\zeta+A_3\theta^2+A_4\theta\zeta+A_5\zeta^2\\
&+A_6\theta^3+A_7\theta^2\zeta+A_8\theta\zeta^2+A_9\zeta^3
+A_{10}\theta^3\zeta+A_{11}\theta\zeta^3,
\end{aligned}
\]
we compute
\[
P_{\theta\theta}
=
2A_3+6A_6\theta+2A_7\zeta+6A_{10}\theta\zeta,
\]
and
\[
P_{\zeta\zeta}
=
2A_5+2A_8\theta+6A_9\zeta+6A_{11}\theta\zeta.
\]
Hence
\[
\Delta P_{ij}
=
\frac{1}{h_m^2}(P_{\theta\theta}+P_{\zeta\zeta})
=
\frac{1}{h_m^2}
\Bigl(
2(A_3+A_5)
+(6A_6+2A_8)\theta
+(2A_7+6A_9)\zeta
+6(A_{10}+A_{11})\theta\zeta
\Bigr),
\]
so $\Delta P_{ij}$ is bilinear in the local variables.

\section{Stability and Consistency Implied by the Interpolation Assumption}\label{app:interp}

Here we summarize what follows from the uniform rescaled regularity assumption in Section~4. These estimates provide the input needed to keep the switch defect small, but they do \emph{not} by themselves prove the global bound
\[
\varepsilon_{m,h_m}^{sw}\le C_{sw}A_m^2h_m^2
\]
unconditionally. In particular, the contributions from interface-crossing stencils and from the boundary/interface strip are absorbed into the switch defect in the main theorem.

\begin{lemma}[Stability of the 12-point interpolation operator]
Let $\mathcal P_{h_m}$ denote the 12-point interpolation operator on the reference cell. Then $\mathcal P_{h_m}$ is linear and bounded: there exists $C_I>0$ such that
\[
\|\mathcal P_{h_m}W_h\|_{L^\infty(\mathrm{cell})}
\le
C_I\max_{(a,b)\in S_{12}}|W_{a,b}|.
\]
\end{lemma}

\begin{proof}
By unisolvence in Section~4, the coefficients on the reference cell are given linearly by a fixed matrix $M^{-1}$ acting on the data vector. Finite-dimensional norm equivalence then bounds both the coefficients and the polynomial values by the maximum norm of the data.
\end{proof}

\begin{lemma}[Consistency for smooth profiles]
Let $Q_{\rm en}\Subset Q_{\rm int}$ and $W\in C^4(Q_{\rm int})$. Construct the 12-point interpolation on the coarse cells contained in $Q_{\rm en}$. Then there exists a constant $C>0$ such that
\[
\|\mathcal P_{h_m}I_{h_m}W-W\|_{L^\infty(\mathrm{cell})}
\le
Ch_m^4\|W\|_{C^4(\omega_{ij})},
\]
and, at fine nodes whose centered Laplacian stencil lies inside a single coarse cell,
\[
\|\Delta \mathcal P_{h_m}I_{h_m}W-\Delta W\|_{L^\infty(\mathrm{cell})}
\le
Ch_m^2\|W\|_{C^4(\omega_{ij})}.
\]
\end{lemma}

\begin{proof}
Let
\[
\mathcal V_{12}
:=
\mathrm{span}\{1,\theta,\zeta,\theta^2,\theta\zeta,\zeta^2,
\theta^3,\theta^2\zeta,\theta\zeta^2,\zeta^3,\theta^3\zeta,\theta\zeta^3\}.
\]
This space is unisolvent for the 12-point data and contains every bivariate polynomial of \emph{total degree at most three}. It is not the full tensor-product cubic space, since for example $\theta^2\zeta^2\notin \mathcal V_{12}$.

Let $T_3W$ be the third-order Taylor polynomial of $W$ about a point in the cell, written in physical variables. Because $T_3W$ has total degree at most three, the interpolation operator reproduces it exactly:
\[
\mathcal P_{h_m}I_{h_m}(T_3W)=T_3W.
\]
Therefore
\[
\mathcal P_{h_m}I_{h_m}W-W
=\mathcal P_{h_m}I_{h_m}(W-T_3W)-(W-T_3W).
\]
The $C^4$ Taylor remainder on the fixed 12-point stencil is bounded by $Ch_m^4\|W\|_{C^4(\omega_{ij})}$, and the stability lemma transfers this bound through $\mathcal P_{h_m}$. This gives the stated $L^\infty$ interpolation error.

For the Laplacian estimate, the same reproduction identity gives
\[
\Delta\bigl(\mathcal P_{h_m}I_{h_m}W-W\bigr)
=\Delta\mathcal P_{h_m}I_{h_m}(W-T_3W)-\Delta(W-T_3W).
\]
Taking two derivatives of a cell polynomial introduces the factor $h_m^{-2}$ relative to the reference-cell coefficients. Hence the $O(h_m^4)$ remainder data yield an $O(h_m^2)$ bound after applying $\Delta$. The term $\Delta(W-T_3W)$ is also $O(h_m^2)\|W\|_{C^4(\omega_{ij})}$. Exact Laplace compatibility is invoked only at fine nodes whose centered stencil stays inside a single coarse cell; interface-crossing stencil contributions are not included in this local consistency estimate.
\end{proof}

\begin{proposition}[Bounds for the prolongation under the uniform regularity assumption]\label{prop:prolongation-bounds}
Assume the uniform rescaled regularity assumption of Section~4. Then, for sufficiently small $h_m$, there exist constants $c_0,C_0,C_4>0$, independent of $m$, such that on the coarse cells of $Q_{\rm en}$,
\[
c_0\le \mathcal P_{h_m}Z_m^{J_m}\le C_0,
\qquad
\|\mathcal P_{h_m}Z_m^{J_m}\|_{C_{\rm pw}^4(Q_{\rm en})}\le C_4.
\]
\end{proposition}

\begin{proof}
Because the interpolation of cells in $Q_{\rm en}$ uses a one-layer stencil, all needed data lie in $Q_{\rm int}$ under the assumption $Q_{\rm en}\Subset Q_{\rm int}$. We decompose
\[
\mathcal P_{h_m}Z_m^{J_m}
=
\mathcal P_{h_m}I_{h_m}W_m^*
+
\mathcal P_{h_m}\bigl(Z_m^{J_m}-I_{h_m}W_m^*\bigr).
\]
The first term is uniformly close to $W_m^*$ on the coarse cells of $Q_{\rm en}$ by the consistency estimate and the uniform $C^4$ bound on $W_m^*$ over $Q_{\rm int}$. The second term is $O(\varepsilon_{h_m})$ by the stability of the interpolation operator together with
\[
\|Z_m^{J_m}-I_{h_m}W_m^*\|_{\ell^\infty(Q_{{\rm int},h_m})}\le \varepsilon_{h_m}.
\]
If $h_m$ and $\varepsilon_{h_m}$ are sufficiently small, this yields uniform positive lower and upper bounds on $Q_{\rm en}$. The piecewise $C^4$ bound follows from the coefficient representation together with the assumed uniform control of discrete differences on the stencil patches contained in $Q_{\rm int}$.
\end{proof}

\begin{remark}
The previous lemmas quantify the raw interpolation error in the single-cell interior. To turn them into a full theorem on the switch defect one would still need to control, in a unified way, the contribution of interface-crossing stencils, the boundary/interface strip near the computational boundary, and, if present, the effect of the energy-repair projection. This is why the bound
\[
\varepsilon_{m,h_m}^{sw}\le C_{sw}A_m^2h_m^2
\]
remains a conditional assumption in the main text.
\end{remark}

\section{Solvability and Local Uniqueness of the Admissible Branch}\label{app:admissible}

The minimizing-movement step may admit several algebraic branches. This appendix shows that, under a sufficiently small time step, there exists an energy-dissipating branch that stays positive and is locally unique.

\begin{theorem}[Solvability and local uniqueness for small time steps]
Let the previous step on the fixed stage $m$ satisfy
\[
Z_{m,ij}^j\ge \eta>0.
\]
Assume moreover that
\[
\Delta s
<
\Delta s_m^*
:=
\min\left\{
\frac{A_m^2h_m^2\eta^2}{8E_{h_m,m}(Z_m^j)},
\frac{\eta^3}{16\lambda}
\right\}.
\]
Then the minimizing-movement functional
\[
J_m(Y)=E_{h_m,m}(Y)+\frac{A_m^2}{2\Delta s}\|Y-Z_m^j\|_{2,h_m}^2
\]
admits a minimizer over $\{Y\ge 0\}$, and every such minimizer $Z_m^{j+1}$ satisfies
\[
Z_{m,ij}^{j+1}\ge \eta/2>0.
\]
Hence it solves the Euler--Lagrange equation
\[
\frac{Z_{m,ij}^{j+1}-Z_{m,ij}^j}{\Delta s}
=
\Delta_{h_m}Z_{m,ij}^{j+1}
-
\frac{\lambda}{(Z_{m,ij}^{j+1})^2K_{h_m,m}(Z_m^{j+1})^2}.
\]
Furthermore, this solution is unique in a neighborhood of $Z_m^j$.
\end{theorem}

\begin{proof}
Consider $J_m$ on the closed convex set $\{Y\ge 0\}$. By the lower-semicontinuous extension adopted in Section~3, the reciprocal term is well defined there, taking the value $0$ on the vanishing branch. The quadratic distance term makes $J_m$ coercive, and finite dimensionality yields a minimizer $Y^*$. By minimality,
\[
E_{h_m,m}(Y^*)+\frac{A_m^2}{2\Delta s}\|Y^*-Z_m^j\|_{2,h_m}^2
\le
E_{h_m,m}(Z_m^j),
\]
so
\[
\|Y^*-Z_m^j\|_{2,h_m}
\le
\sqrt{\frac{2\Delta s}{A_m^2}E_{h_m,m}(Z_m^j)}.
\]
On a two-dimensional grid,
\[
\|R\|_{\ell^\infty(Q_{{\rm en},h_m})}\le h_m^{-1}\|R\|_{2,h_m},
\]
and the condition
\[
\Delta s<\frac{A_m^2h_m^2\eta^2}{8E_{h_m,m}(Z_m^j)}
\]
therefore implies
\[
\|Y^*-Z_m^j\|_{\ell^\infty(Q_{{\rm en},h_m})}<\eta/2.
\]
Thus $Y_{ij}^*\ge \eta/2$, so the minimizer lies in the interior of the positive cone, and the Euler--Lagrange equation follows.

Local uniqueness follows from the Hessian estimate in Appendix~\ref{app:details}. In a neighborhood where $Y\ge \eta/2$,
\[
D^2J_m(Y)[\Phi,\Phi]
\ge
A_m^2\left(\frac{1}{\Delta s}-\frac{16\lambda}{\eta^3}\right)\|\Phi\|_{2,h_m}^2.
\]
If additionally
\[
\Delta s<\frac{\eta^3}{16\lambda},
\]
then $J_m$ is strictly convex in that neighborhood. Hence the admissible branch is locally unique.
\end{proof}

\begin{remark}
What is proved here is the existence and local uniqueness of the admissible branch that is continuously connected to the previous step. We do \emph{not} claim global uniqueness among all positive solutions of the nonlinear implicit equation.
\end{remark}

\end{document}